\documentclass[11pt]{article}

\usepackage{titling}

\usepackage{tikz}
\usetikzlibrary{matrix}
\usetikzlibrary{patterns}

\usepackage{amsmath}
\usepackage{amsfonts}
\usepackage{amssymb}
\usepackage{amsthm}
\usepackage{enumerate}
\usepackage{geometry}
\title{\bf An analogue of Ryser's Theorem for partial Sudoku squares}
\author{{\Large P. J. Cameron} \thanks{School of Mathematical Sciences, Queen Mary, University of London, Mile End Road, London E1 4NS, U.K. Email: {\tt p.j.cameron@qmul.ac.uk} }
\and {\Large A. J. W. Hilton} \thanks{Department of Mathematics and Statistics, University of Reading, Whiteknights, Reading RG6 6AX, U.K. \textit{and} School of Mathematical Sciences, Queen Mary, University of London, Mile End Road, London E1 4NS, U.K. Email: {\tt a.j.w.hilton@reading.ac.uk}}
\and {\Large E. R. Vaughan} \thanks{School of Mathematical Sciences, Queen Mary, University of London, Mile End Road, London E1 4NS, U.K. Email: {\tt e.vaughan@qmul.ac.uk}}
}

\date{\today}

\renewcommand{\maketitlehookd}{\emph{In memory of Ralph Stanton and his tremendous work for the subject of Combinatorics.}}

\newtheorem{theorem}{Theorem}
\newtheorem{lemma}[theorem]{Lemma}

\theoremstyle{definition}
\newtheorem{problem}{Problem}

\providecommand{\size}[1]{\left|#1\right|}
\providecommand{\floor}[1]{\lfloor#1\rfloor}
\providecommand{\ceiling}[1]{\lceil#1\rceil}

\providecommand{\col}{\mathcal{C}}
\providecommand{\hc}{Hall's Condition}
\providecommand{\hi}{Hall Inequality}
\providecommand{\pqs}{$(p,q)$-Sudoku}
\providecommand{\rts}{$r \times s$}
\providecommand{\ptq}{$p \times q$}
\providecommand{\ntn}{$n \times n$}
\providecommand{\otn}{$\{1, 2, \dots, n\}$}

\begin{document}

\maketitle

\begin{abstract}
In 1956 Ryser gave a necessary and sufficient condition for a partial latin rectangle to be completable to a latin square. In 1990 Hilton and Johnson showed that Ryser's condition could be reformulated in terms of \hc\ for partial latin squares. Thus Ryser's Theorem can be interpreted as saying that any partial latin rectangle $R$ can be completed if and only if $R$ satisfies \hc\ for partial latin squares.

We define \hc\ for partial Sudoku squares and show that \hc\ for partial Sudoku squares gives a criterion for the completion of partial Sudoku rectangles that is both necessary and sufficient. In the particular case where $n=pq$, $p|r$, $q|s$, the result is especially simple, as we show that any $r \times s$ partial \pqs\ rectangle can be completed (no further condition being necessary).
\end{abstract}

\section{Introduction}

A \textit{latin square of order $n$} is an $n \times n$ array with entries from \otn\ in which each symbol occurs exactly once in each row and column. An \rts\ latin rectangle of order $n$ is an \rts\ array in which each cell is filled from the set \otn\ of symbols and no symbol occurs more than once in any row or column. In 1956 Ryser \cite{ryser} proved the following well-known theorem.

\begin{theorem} Let $R_l$ be an \rts\ latin rectangle on the symbols \otn. Then $R_l$ can be completed to form an \ntn\ latin square on the symbols \otn\ if and only if
\begin{align*}
N(i) \ge r + s - n 
\end{align*} for all $1 \le i \le n$, where $N(i)$ is the number of times the symbol $i$ occurs in $R_l$. \label{ryser} \end{theorem}

\providecommand{\onedash}{Theorem $\mathrm{\ref{ryser}'}$}
\providecommand{\onedashitalic}{Theorem $\ref{ryser}'$}
\providecommand{\onedashbold}{Theorem $\mathbf{\ref{ryser}'}$}

\newtheorem*{theorem1dash}{\onedashbold}

If $n=pq$, a \textit{\pqs}\ square is a latin square of order $n$ which is partitioned into \ptq\ rectangles, each containing a set
\[ \left\{ (xp + i, yq + j) \; :  \; 1 \le i \le p , \, 1 \le j \le q \right\} \]
of cells for some $x$ and $y$, where $0 \le x \le q - 1$, $0 \le y \le p - 1$, and which have the property that each of these \ptq\ rectangles contains each of the symbols $1, 2, \dots, n$ exactly once. (See Figure \ref{square}.)

\begin{figure}
\begin{center} \begin{tikzpicture}
\begin{scope}[scale=1.4]
\draw[gray,xstep=12mm, ystep=9mm] (0,0) grid (3.6, 3.6);
\draw[black, <->] (0,-0.2) -- node[below] {$n$} (3.6,-0.2);
\draw[black, <->] (3.8,0) -- node[right] {$n$} (3.8,3.6);
\draw[black, <->] (0,3.4) -- node[below] {$q$} (1.2,3.4);
\draw[black, <->] (1.4,3.6) -- node[right] {$p$} (1.4,2.7);
\end{scope}
\end{tikzpicture} \end{center} \caption{A \pqs\ square.} \label{square}
\end{figure}

The puzzle game \emph{Sudoku} was invented by Howard Garns in 1979 \cite{bcc}. The puzzle consists of a partial Sudoku square, in which a few of the cells are preassigned, and the objective is to complete the partial Sudoku square to a Sudoku square. Usually the puzzles found in newspapers and magazines use $(3,3)$-Sudoku squares, although occasionally other types are used, for example $(2,3)$- or $(3,4)$-Sudoku squares.

An \rts\ \pqs\ rectangle on $n=pq$ symbols is an \rts\ latin rectangle which is partitioned into \ptq\ rectangles plus, if $p \nmid r$ and $q \nmid s$, along the right hand border, rectangles of size $p \times (s - s^*)$ and along the lower border, rectangles of size $(r - r^*) \times q$, and finally a rectangle of size $(r - r^*) \times (s - s^*)$, where $r^* = \floor{r/p}p$ and $s^* = \floor{s/q}q$, as illustrated in Figure \ref{rectangle}.

Our first analogue for Sudoku squares of Ryser's Theorem is the following simple result.

\begin{theorem} Let $p|r$, $q|s$ and $n=pq$. Any \rts\ \pqs\ rectangle $R_s$ can be extended to an \ntn\ \pqs\ square. \label{simpleryser} \end{theorem}

Thus there is no let or hindrance to the extension; in this respect Theorem \ref{simpleryser} is simpler than Ryser's Theorem.

The second analogue of Ryser's Theorem is necessarily a bit more complicated than Ryser's Theorem itself. In order to describe it we first define some associated bipartite graphs $S_\alpha$ for $\alpha = 1, 2, \dots, (r^*/p) + 1$ and $B_\beta$ for $\beta = 1, 2, \dots, (s^*/q) + 1$. Here $S$ stands for ``side'' (of the rectangle), and $B$ stands for ``bottom'' (of the rectangle). These graphs are only defined if $p \nmid r$ or $q \nmid s$ respectively.

Firstly, if $p \nmid r$, graphs $S^*_\alpha$ for $\alpha = 1, \dots, r^*/p$ are defined with vertices
\[ v_{(\alpha - 1)p + 1}, \dots, v_{\alpha p} \]
(corresponding to the rows $\rho_{(\alpha - 1)p + 1}, \dots, \rho_{\alpha p}$) and vertices $w_1, \dots, w_n$ (corresponding to the symbols $1, \dots, n$). The edge $v_{(\alpha - 1)p + i} w_j$ is placed in $S^*_\alpha$ if symbol $j$ does not occur in row $\rho_{(\alpha - 1)p + i}$ of $R_s$. Then, for each $\alpha$, $1 \le \alpha \le r^*/p$, from $S^*_\alpha$ we construct $S_\alpha$ by replicating vertex $v_i$ $q - (s - s^*)$ times, the replications being denoted by $v_{ij}$ ($1 \le j \le q - (s - s^*)$), and by joining each replicated vertex of $v_i$ to each of the vertices $w_1, \dots, w_n$ that was joined to $v_i$ in $S^*_\alpha$.

\begin{figure}
\begin{center} \begin{tikzpicture}
\begin{scope}[scale=1.4]
\draw[gray,step=10mm] (0,-0.5) grid (4.5, 3) (0,-0.5) -- (4.5,-0.5) -- (4.5, 3);
\draw[black, <->] (0,3.2) -- node[above] {$s^*$} (4,3.2);
\draw[black, <->] (0,-0.7) -- node[below] {$s$} (4.5,-0.7);
\draw[black, <->] (4.7,-0.5) -- node[right] {$r$} (4.7,3);
\draw[black, <->] (-0.2,0) -- node[left] {$r^*$} (-0.2,3);
\draw[black, <->] (0,2.8) -- node[below] {$q$} (1,2.8);
\draw[black, <->] (1.2,2) -- node[right] {$p$} (1.2,3);
\end{scope}
\end{tikzpicture} \end{center} \caption{An \rts\ \pqs\ rectangle.} \label{rectangle}
\end{figure}

For $\alpha = (r^*/p) + 1$, $S^*_\alpha$ is defined similarly. The vertices are
\[ v_{(\alpha - 1)p + 1} = v_{r^* + 1}, \dots, v_r \]
(corresponding to the rows $\rho_{r^* + 1}, \dots, \rho_r$) and vertices $w_1, \dots, w_n$. An edge $v_{r^* + i} w_j$ is placed in $S^*_\alpha$ if symbol $j$ does not occur in row $\rho_{r^* + i}$. From $S^*_\alpha$ we construct $S_\alpha$ by replicating vertex $v_i$ $q - (s - s^*)$ times, the replicated vertices being denoted by $v_{ij}$ ($1 \le j \le q - (s - s^*)$), and by joining each replicated vertex to each of the vertices $w_1, \dots, w_n$ that was joined to $v_i$ in $S^*_\alpha$.

If $q \nmid s$, the bipartite graphs $B_\beta$ ($1 \le \beta \le (s^*/q) + 1$) are defined in a similar way, but with columns and rows being interchanged, and with $p$, $r$, $r^*$ interchanged with $q$, $s$, $s^*$.

We are now able to state the second, more complicated analogue of Ryser's Theorem.

\begin{theorem} Let $n=pq$. An \rts\ \pqs\ rectangle $R_s$ can be extended to an \ntn\ \pqs\ square if and only if, when $r \nmid p$, for $1 \le \alpha \le \ceiling{r/q}$, the graph $S_\alpha$ has a matching from the replicated row vertices into the symbol vertices, and, when $s \nmid q$, for $1 \le \beta \le \ceiling{s/q}$, the graph $B_\beta$ has a matching from the replicated column vertices into the symbol vertices. \label{complicatedryser} \end{theorem}

\providecommand{\threedash}{Theorem $\mathrm{\ref{complicatedryser}'}$}
\providecommand{\threedashitalic}{Theorem $\ref{complicatedryser}'$}
\providecommand{\threedashbold}{Theorem $\mathbf{\ref{complicatedryser}'}$}

\newtheorem*{theorem3dash}{\threedashbold}

We have not been able to express these two conditions in the very simple way that Ryser's Theorem is expressed. Of course the existence of a matching can be expressed in terms of the satisfaction of a collection of inequalities---as Hall showed \cite{hall} in 1935. (In this paper we reserve the terms ``\hc'' and ``Hall Inequalities'' for three more general analogous situations, one for partially coloured graphs, one for partial latin squares, and the third for partial Sudoku squares.)

We now turn to the definitions of \hc\ for graphs, \hc\ for partial latin squares, and \hc\ for partial Sudoku squares.

Let $G$ be a finite simple graph. Let $\col$ be a set of colours, and let $L : V(G) \rightarrow 2^\col$, the family of subsets of $\col$, be a \textit{list assignment} to $G$, so that each vertex is assigned a ``list'' or finite set of colours from $\col$. A function $f : V(G) \rightarrow \col$ is a \textit{vertex colouring} of $G$ provided that adjacent vertices have different colours. An \textit{$L$-colouring} of $V(G)$ is a vertex colouring $f : V(G) \rightarrow \col$ such that $f(v) \in L(v)$ for each $v \in V(G)$. An $L$-colouring is often called a \textit{list-colouring}. An \textit{independent set} $I$ of vertices of $G$ is a set of vertices which does not contain any neighbouring pair of vertices. We let $\alpha(L, \sigma, G)$ be the maximum number of vertices in an independent set of vertices of $G$, each of which contains $\sigma$ in its list. A necessary condition for there to exist an $L$-colouring of $G$ is that
\begin{equation}
\sum_{\sigma \in \col} \alpha(L, \sigma, H) \ge \size{V(H)} \tag{$*$} \label{histar} 
\end{equation}
for each subgraph $H$ of $G$. For a given subgraph $H$ of $G$, inequality \eqref{histar} is called \textit{Hall's Inequality for $H$}, and the condition that ``for each subgraph $H$ of $G$, the Hall Inequality for $H$ is satisfied'' is called \textit{\hc\ for $G$}.

In the special case where $G$ is a complete graph, \hc\ is sufficient to ensure that $H$ has an $L$-colouring. (This is Hall's Theorem \cite{hall}, and an $L$-colouring in this case is a system of distinct representatives.) \hc\ for graphs has been the subject of several recent papers (see \cite{cropper00, cropper02, daneshgar, deh, eslahchijohnson, hiltonjohnson90a, hiltonjohnson90b, hiltonjohnson99, hjw, hiltonvaughan, hallstrength}). In general, \hc\ for graphs is not sufficient to ensure there is a list-colouring. 

A latin square can be thought of as a proper vertex colouring of the cartesian product graph $K_n \square K_n$ with $\chi(K_n \square K_n) = n$ colours. This graph can be visualised as an \ntn\ grid in which the $n$ vertices in each row are joined to each other, as are the $n$ vertices in each column. However instead of giving a discussion in terms of graphs (which is easy enough to do) we shall give a more direct discussion using the more usual terminology for latin squares.

In a partial \ntn\ latin square $P_l$, some cells are filled from a set $\col$ of $n$ symbols and some are empty. We can define a list $L_l(v)$ for each cell $v$ as follows:

\[L_l (v) = \left\lbrace \begin{minipage}{9cm}
the symbol in cell $v$, if cell $v$ is filled; \vspace{2mm}

the set of symbols which are not used in the cells in the same row as $v$, nor in the cells in the same column as $v$, if cell $v$ is empty.
\end{minipage}\right.\]

A set of cells of $P_l$ is called \textit{independent} if no two cells occur in the same row or column. For a subset $Q$ of cells of $P_l$, we let $\alpha(L_l, \sigma, Q)$ be the maximum number of independent cells of $Q$ all of which contain $\sigma$ in their lists. \textit{Hall's Inequality} for a set $Q$ of cells of $P_l$ is that
\[ \sum_{\sigma \in \col} \alpha(L_l, \sigma, Q) \ge \size{Q},\]
where $\col$ is the set of $n$ symbols available to be used in $P_l$. \textit{\hc}\ is the collection of all $2^{n^2}$ inequalities obtained by letting $Q$ range over all subsets of $P_l$.

Hilton and Johnson in 1990 in \cite{hiltonjohnson90a} showed that Ryser's Theorem could be restated in the following way.

\begin{theorem1dash} Let $R_l$ be an \rts\ latin rectangle on the symbols \otn. Let $P_l$ be a partial \ntn\ latin square with $R_l$ in the top left-hand corner, and the remaining cells empty. Then $P_l$ can be completed to an \ntn\ latin square if and only if $P_l$ satisfies \hc\ for partial latin squares. \end{theorem1dash}

In this case we do not need the whole set of inequalities comprising \hc. In fact if the \hi\ for the whole \ntn\ partial latin square $P_l$ is satisfied then the partial latin square can be completed. Thus the $n$ inequalities of Ryser's Theorem can be replaced by just one inequality.

Before turning to \hc\ for Sudoku squares it is convenient to define Gerechte Designs and \hc\ for them. Let $(P_1, P_2, \dots, P_n)$ be a partition of the set of cells of an \ntn\ latin square in which $\size{P_i} = n$ for each $1 \le i \le n$. An \ntn\ latin square is a \textit{$(P_1, P_2, \dots, P_n)$-Gerechte Design} if each $P_i$ contains each symbol exactly once. In the case when $n=pq$ and each $P_i$ is a \ptq\ rectangle, then the Gerechte Design is a \pqs\ square.

From a graph theory point of view, a $(P_1, P_2, \dots, P_n)$-Gerechte Design can be thought of as a proper vertex colouring with $n$ colours of the graph obtained from the cartesian product graph $K_n \square K_n$ (the vertices here correspond to the cells, and the vertex sets of the first and second $K_n$'s correspond to the rows and columns respectively) by placing an edge between two vertices if the corresponding cells are in the same part $P_i$.

We define \hc\ for partial $(P_1, P_2, \dots, P_n)$-Gerechte Designs as follows. If $P_D$ is a partial $(P_1, P_2, \dots, P_n)$-Gerechte Design, for each cell $r$ we define a list in the following way.

\[L_D(v) = \left\lbrace \begin{minipage}{9cm}
the symbol in the cell $v$ of $P_D$, if cell $v$ is filled; \vspace{2mm}

the set of symbols which are not used in the cells in the same row as $v$, nor in same column as $v$, nor in the same part $P_i$ of the partition containing $v$.
\end{minipage}\right.\]

A set of cells of $P_D$ is said to be \textit{independent} if no two cells occur in the same row, or the same column, or the same part. For a subset $Q$ of the cells we let $\alpha(L_D, \sigma, Q)$ be the maximum number of independent cells of $Q$ all of which contain the symbol $\sigma$ in their lists. \textit{Hall's Inequality for $Q$} is
\[ \sum_{\sigma \in \col} \alpha(L_D, \sigma, Q) \ge \size{Q}, \]
and if Hall's Inequality is satisfied for each subset $Q$ of $P_D$, then $P_D$ is said to satisfy \textit{\hc\ for partial $(P_1, P_2, \dots, P_n)$-Gerechte Designs}, or, briefly, \textit{\hc\ for Gerechte Designs}.

Now we turn to discuss \hc\ for partial Sudoku squares. First let us note that we shall describe the \ptq\ subrectangles forming the \pqs\ square as the \textit{big (Sudoku) cells}, and the $1 \times 1$ subrectangles as the \textit{small (Sudoku) cells}. A \pqs\ square on a set $\col$ of $n=pq$ symbols can be thought of as a proper vertex colouring of a graph obtained from the cartesian product graph $K_n \square K_n$, in which, visualizing $K_n \square K_n$ as a grid, as described earlier, the \ptq\ subrectangles (i.e. the big cells) correspond to complete bipartite subgraphs $K_{p,q}$. If the first $K_n$ has vertices
\[ v_1, \dots, v_p, v_{p+1}, \dots, v_{2p}, \dots, v_{(q-1)p + 1}, \dots, v_{qp}, \]
 and the second $K_n$ has vertices
\[ w_1, \dots, w_q, w_{q+1}, \dots, w_{2q}, \dots, w_{(p-1)q + 1}, \dots, v_{pq}, \]
then the big Sudoku cells correspond to the $K_{p,q}$ on the vertices
\[v_{(x-1)p+1}, \dots, v_{xp} \qquad \text{and} \qquad w_{(y-1)q+1}, \dots, w_{yq}\]
($x = 1, 2, \dots, q$; $y = 1, 2, \dots, p$). In this case, for each small cell $v$, we define

\[L_S(v) = \left\lbrace \begin{minipage}{9cm}
the symbol in the cell $v$, if cell $v$ is filled; \vspace{2mm}

the set of symbols which are not used in the small cells in the same row as $v$, nor in small cells in the same column as $v$, nor in the big \ptq\ Sudoku cell containing $v$.
\end{minipage}\right.\]


A set of cells of $P_S$ is said to be \textit{independent} if no two cells occur in the same row, or the same column, or the same big cell. For a subset $Q$ of the cells we let $\alpha(L_S, \sigma, Q)$ be the maximum number of independent cells of $Q$ all of which contain the symbol $\sigma$ in their lists. \textit{Hall's Inequality for $Q$} is
\[ \sum_{\sigma \in \col} \alpha(L_S, \sigma, Q) \ge \size{Q}, \]
and if Hall's Inequality is satisfied for each subset $Q$ of $P_S$, then $P_S$ is said to satisfy \textit{\hc\ for partial Sudoku squares}.

We shall show that Theorem \ref{complicatedryser} (together with Theorem \ref{simpleryser}) can be reformulated as follows.

\begin{theorem3dash} Let $n=pq$. Let $R_S$ be an \rts\ \pqs\ rectangle on the symbol set $\col = \{1, 2, \dots, n\}$. Let $P_S$ be a partial \pqs\ square with $R_S$ in the top left-hand corner, and the remaining cells empty. Then $P_S$ can be completed to an \ntn\ \pqs\ square if and only if $P_S$ satisfies \hc\ for partial Sudoku squares. \end{theorem3dash}

\begin{figure}
\begin{center} \begin{tikzpicture}
\begin{scope}[scale=0.95]
\fill[gray!50] (0,0) rectangle (8.4,0.225)
(0,2.7) rectangle (8.4,3.6) (2.4,4.5) rectangle (3.6,5.4)
(7.2,2.025) rectangle (8.4,2.25);
\draw[gray,xstep=12mm, ystep=9mm] (0,0) rectangle (8.4, 5.4)
(0,0) grid (3.6,5.4)
(0,1.8)--(8.4,1.8) (0,2.7)--(8.4,2.7) (0,3.6)--(8.4,3.6)
(4.8,5.4)--(4.8,1.8)
(3.6,2.025)--(8.4,2.025) (3.6,2.25)--(8.4,2.25) (3.6,2.475)--(8.4,2.475) 
(0,0.225)--(8.4,0.225)
(4.8,1.8)--(4.8,2.7) (6,1.8)--(6,2.7) (7.2,1.8)--(7.2,2.7)
;
\draw[black] (0,0) rectangle (3.6,5.4) (1.8,2.25) node {$S$};
\draw[black, <->] (3.6,1.6) -- node[below] {$q$} (4.8,1.6);
\draw[black, <->] (0,5.6) -- node[above] {$s$} (3.5,5.6);
\draw[black, <->] (-0.2,5.4) -- node[left] {$r$} (-0.2,0);
\draw[black, <->] (0.2,2.7) -- node[right] {$p$} (0.2,3.6);
\draw[black, <-] (8.4,0.1125)--(9,0.1125) node[right] {small row};
\draw[black, <-] (8.4,2.1375)--(9,2.1375) node[right] {medium cell};
\draw[black, <-] (8.4,3.15)--(9,3.15) node[right] {big row};
\draw[black, <-] (3.6,4.95)--(5,4.95) node[right] {big cell};

\end{scope}
\end{tikzpicture} \end{center} \caption{Big and medium cells, big and small rows.} \label{bigmediumsmall}
\end{figure}
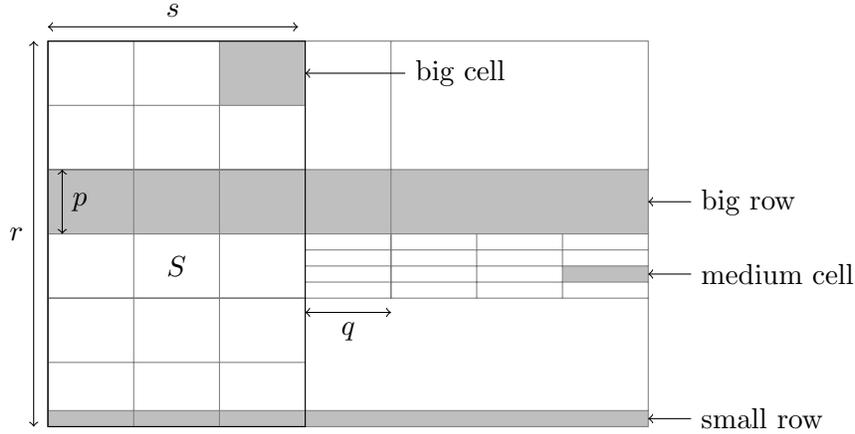

\section{Partial Sudoku Squares: the simple case}

In this section we prove Theorem \ref{simpleryser}. First we need some preliminaries.

A \textit{$k$-edge-colouring} of a finite multigraph $G$ is a map
\[\phi : E(G) \rightarrow \{1, 2, \dots, k\},\]
where $E(G)$ is the edge set of $G$. For any vertex $v \in V(G)$, the vertex set of $G$, and any $i \in \{1, \dots, k\}$, let $E_i(v)$ be the set of edges incident with $v$ of colour $i$. A $k$-edge-colouring is called \textit{equitable} if
\[ \big| \size{E_i(v)} - \size{E_j(v)} \big| \le 1 \; \text{for all $v \in V(G)$ and $i, j \in \{1, 2, \dots, k\}$.} \]

The following lemma is easy to prove; it is due independently to McDiarmid \cite{mcdiarmid} and de Werra \cite{dewerra}.

\begin{lemma} Let $k \ge 1$ be an integer, and let $G$ be a bipartite multigraph. Then $G$ has an equitable $k$-edge-colouring. \label{equitable} \end{lemma}

Let $L$ be a latin square of order $n$, and let $S = (p_1, \dots, p_s)$, $T = (q_1, \dots, q_t)$ and $U = (r_1, \dots, r_u)$ be three compositions of $n$. (A \emph{composition} of $n$ is an ordered set of positive integers which sum to $n$.) An \textit{$(S,T,U)$-amalgamation} of $L$ is an $s \times t$ array constructed as follows. We place a symbol $k$ in cell $(i,j)$ every time that a symbol from
\[ \{ r_1 + \cdots + r_{k-1} + 1, \dots, r_1 + \cdots + r_k\} \]
occurs in one of the cells $(\alpha, \beta)$ where
\[ \alpha \in \{ p_1 + \cdots + p_{i-1} + 1, \dots, p_1 + \cdots + p_i\} \]
and
\[ \beta \in \{ q_1 + \cdots + q_{j-1} + 1, \dots, q_1 + \cdots + q_j\}. \]

Let $n \ge 1$ be an integer and let $S = (p_1, \dots, p_s)$, $T = (q_1, \dots, q_t)$ and $U = (r_1, \dots, r_u)$ be three compositions of $n$. An \textit{$(S,T,U)$-outline latin square} of order $n$ is an $s \times t$ array containing $n^2$ symbols from $\{1, \dots, u\}$ (each cell may contain more than one symbol), such that the following conditions hold:

\begin{enumerate}[(i)]
\item row $i$ contains symbol $k$ $p_i r_k$ times;
\item column $j$ contains symbol $k$ $q_j r_k$ times;
\item cell $(i,j)$ contains $p_i q_j$ symbols, counting repetitions.
\end{enumerate}

It is easy to check that the $(S,T,U)$-amalgamation of a latin square of order $n$ is an $(S,T,U)$-outline latin square of order $n$. The following theorem, sometimes called the \textit{Amalgamated Latin Square Theorem}, shows that the converse is also true. It was proved by the second author in \cite{hilton80} in 1980. See \cite{andersenhilton1} and \cite{andersenhilton2} for earlier work by Andersen and the second author on somewhat similar lines, and \cite{hilton87} for a clearer account of the Amalgamated Latin Square Theorem.

\begin{theorem} Let $O$ be an $(S,T,U)$-outline latin square. Then there is a latin square $L$ such that $O$ is an $(S,T,U)$-amalgamation of $L$. \label{alst} \end{theorem}

We now turn to the proof of Theorem \ref{simpleryser}, and show that if $n=pq$, $p|r$ and $q|s$, and if $R_S$ is an \rts\ \pqs\ rectangle on the symbols $1, 2 \dots, n$, then the partial \ntn\ Sudoku square containing $R_S$ in its top left-hand corner, the other cells being empty, can be completed to form an \ntn\ \pqs\ square.

Recall that we call the \ptq\ rectangles (some empty, some filled) of our partial \pqs\ square $P_S$ the big cells, and the $1 \times 1$ cells of $P_S$ the small cells. The big cells are arranged in rows which we shall call \textit{big rows}, and in columns which we shall call \textit{big columns}. Each of the $r/p$ big rows of $P_S$ consists of $s/q$ big cells which are filled, and $(n-s)/q$ big cells which are empty. Similarly with the first $s/q$ big columns.

We start by placing the symbols $1, 2, \dots, n$ in each of the big cells which are empty. Then, for the first $r/p$ big rows, working big row by big row, one after the other, we divide each of the big cells carefully into $p$ sets of $q$ small cells, all in the same small row. We shall call these sets \textit{medium cells}. (See Figure \ref{bigmediumsmall}.)

\begin{figure}
\begin{center} \begin{tikzpicture}
\fill[gray!50] (0,0) rectangle (2.5, 4.8);
\draw[gray] (0,1.6)--(6,1.6) (0,3.2)--(6,3.2) (2,0)--(2,4.8) (4,0)--(4,4.8) (2.5,0)--(2.5,4.8)
(2.5,1.8)--(4,1.8) (2.5,2.0)--(4,2.0) (2.5,2.2)--(4,2.2) (2.5,2.4)--(4,2.4) (2.5,2.6)--(4,2.6)
(2.5,2.8)--(4,2.8) (2.5,3.0)--(4,3.0);
\draw[black, <-] (4,2.4)--(5,2.4) node[right] {medium cells to be filled};
\end{tikzpicture} \end{center} \caption{One of the big cells in the $((s^*/q) + 1)$-th big column.} \label{oneofthebig}
\end{figure}
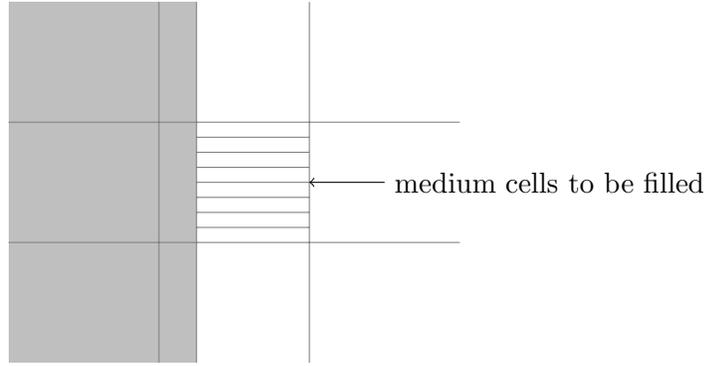

Our first task will be to distribute the $n$ symbols in the previously empty big cells in each of the first $r/p$ big rows amongst the medium cells, so that each medium cell contains $q$ symbols. The only constraint here is that each small row should contain each symbol exactly once. We do this one big row at a time. For this we construct a bipartite graph $G$ as follows. One vertex set contains $p$ vertices, say $r_1, \dots, r_p$, representing the $p$ small rows of some big row. The other vertex set contains $n$ vertices $w_1, \dots, w_n$ representing the symbols. The symbol vertex $w_i$ is joined to the row vertex $v_j$ by an edge if symbol $i$ is not contained in row $j$ of $S$. Each symbol vertex $w_j$ has degree $p-(s/q) = (n-s)/q$ and each row vertex $v_j$ has degree $n-s$. We give $G$ an equitable $k$-edge-colouring with colours $(s/q) + 1, \dots, p$. If symbol vertex $w_i$ is joined to row vertex $v_j$ by an edge coloured $k$, then we place symbol $i$ in row $j$ in the $k$-th big column. Since the edge-colouring is equitable, $q$ symbols are each placed in the $j$-th medium cell in big column $k$.

The analogous process is performed with the first $s/q$ big columns, and in this case the medium cells are vertical and contain $p$ symbols.

At this point it is easy to check that we have an $(S,T,U)$-outline latin square, where
\[ S = (1, 1, \dots, 1, p, p, \dots, p) \]
\[ T = (1, 1, \dots, 1, q, q, \dots, q) \]
\[ U = (1, 1, \dots, 1) \]
where $S$ has $r$ $1$'s and $(n-s)/q$ $q$'s. By Theorem \ref{alst}, there is a latin square $L$ of order $n$ of which our outline latin square is the $(S,T,U)$-amalgamation. Because each big cell was filled with the symbols $1, 2, \dots, n$, this latin square is in fact a Sudoku square. \qed

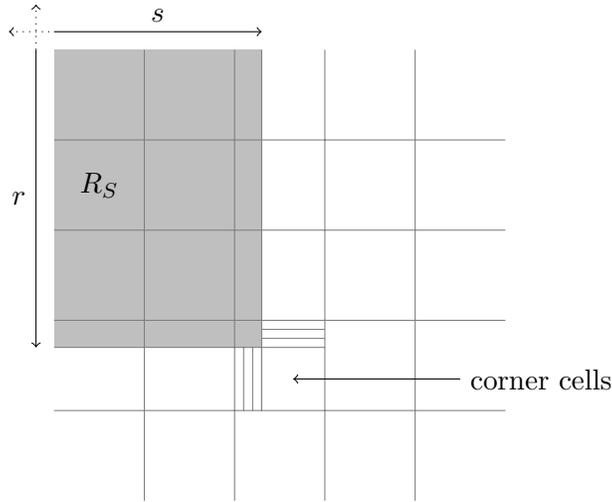
\begin{figure}
\begin{center} \begin{tikzpicture}
\begin{scope}[scale = 1.2]
\fill[gray!50] (0,1.7) rectangle (2.3, 5);
\draw[gray] (1,0)--(1,5) (2,0)--(2,5) (3,0)--(3,5) (4,0)--(4,5) (2.3,1)--(2.3,5)
(0,1)--(5,1) (0,2)--(5,2) (0,3)--(5,3) (0,4)--(5,4) (0,1.7)--(3,1.7)
(2.1,1)--(2.1,1.7) (2.2,1)--(2.2,1.7) (2.3,1)--(2.3,1.7)
(2.3,1.8)--(3,1.8) (2.3,1.9)--(3,1.9);
\draw[black] (0.5,3.5) node {$R_S$};
\draw[black, <-] (2.65,1.35)--(4.5,1.35) node[right] {corner cells};
\draw[black, <-] (-0.2,1.7) -- node[left] {$r$} (-0.2,5);
\draw[black, dotted, ->] (-0.2,5)--(-0.2,5.5);
\draw[black, ->] (0,5.2) -- node[above] {$s$} (2.3,5.2);
\draw[black, dotted, <-] (-0.5,5.2)--(0,5.2);

\end{scope}
\end{tikzpicture} \end{center} \caption{The partition of the corner cell.} \label{cornercell}
\end{figure}

\section{Partial Sudoku squares: the general case.}

In this section we prove Theorem \ref{complicatedryser}. Recall that, in this case, we shall assume that $p \nmid r$ or $q \nmid s$ (or both). Suppose $q \nmid s$. Then graphs $S_\alpha$ are defined for $\alpha = 1, 2, \dots, (s^*/q) + 1$ ($= \ceiling{s/q}$). The tactic for the proof is to apply Theorem \ref{alst} in much the same way as in the proof of Theorem \ref{simpleryser}, except that we have also to consider the first $\ceiling{r/p}$ partly filled big cells in the $((s^*/q) + 1)$-th big column, and, if $p \nmid r$, the first $(s^*/q) + 1$ partly filled big cells in the $((r^*/p) + 1)$-th big row.

\begin{proof}[Proof of Theorem \ref{complicatedryser}.]

We need to divide the unfilled part of the partly filled big cells into horizontally placed medium cells of size $1 \times (q - (s - s^*))$ for the first $r^*/p$ big cells (see Figure \ref{oneofthebig}), and, if $p \nmid r$, just the first $r - r^*$ horizontal medium cells in big cell $((r^*/p) + 1, (s^*/q) + 1)$. (See Figure \ref{cornercell}.) We need to assign the symbols that have not been used in the filled part of these big cells between the horizontal medium cells in such as way that no symbol occurs more than once in any small row. Once these medium cells are filled in the $((s^*/q) + 1)$-th big column, then we can fill the empty horizontal medium cells in the first $r^*/p$ big rows and the first $(r - r^*)$ rows in the $((r^*/p) + 1)$-th big row and in big columns $(s^*/q) + 2, \dots, p$ by applying the argument in the proof of Theorem~\ref{simpleryser}. This argument is not precisely the same because of the need to fill the horizontal medium cells in the small rows $r^* + 1, \dots, r$, but this slight difference presents no difficulties.

The matchings in the graphs $S_\alpha$, for $\alpha = 1, 2, \dots, (r^*/p) + 1$, are used to fill the horizontal medium cells in the $((s^*/q) + 1)$-th big column.

\begin{figure}
\begin{center} \begin{tikzpicture}
\begin{scope}[scale=0.85]
\draw[gray] (0.5,2.5) -- (3,3);
\draw[gray, loosely dotted] (0.5,7)--(0.5,6) (0.5,4)--(0.5,5) (0.5,0)--(0.5,1) (3,0)--(3,6.5) (0.5,3)--(0.5,2);

\fill[black] (0.5,0) circle (2pt) (0.5,1) circle (2pt)
(0.5,4) circle (2pt) (0.5,5) circle (2pt)
(0.5,2.5) circle (2pt)
(0.5,6) circle (2pt) (0.5,7) circle (2pt)

(3,0) circle (2pt) node[right] {$w_n$} (3,7) circle (2pt) node[right] {$w_1$}
(3,6.5) circle (2pt) node[right] {$w_2$}
(3,3) circle (2pt) node[right] {$w_j$};

\draw[black] (-3.5,0) node[right] {$v_{\alpha p,\,q - (s - s^*)}$} (-3.5,1) node[right] {$v_{\alpha p,\,1}$}
(-3.5,4) node[right] {$v_{(\alpha - 1)p + 2,\,q - (s - s^*)}$} (-3.5,5) node[right] {$v_{(\alpha - 1)p + 2,\,1}$}
(-3.5,6) node[right] {$v_{(\alpha - 1)p + 1,\,q - (s - s^*)}$} (-3.5,7) node[right] {$v_{(\alpha - 1)p + 1,\,1}$}
(-3.5,2.5) node[right] {$v_{(\alpha - 1)p + i,\,k}$};
\matrix [matrix of math nodes, left delimiter = \{, left, row sep=9mm] at (-2.7,6.5) {\  \\ \ \\ } ;
\matrix [matrix of math nodes, left delimiter = \{, left, row sep=9mm] at (-2.7,4.5) {\  \\ \ \\ } ;
\matrix [matrix of math nodes, left delimiter = \{, left, row sep=9mm] at (-2.7,0.5) {\  \\ \ \\ } ;

\draw[black] (0,-1) node[below] {the graph $S_\alpha$}
(7, -1) node[below, text width=40mm] {The $i$-th big cell in \\
big column $(s^*/q) + 1$.};
\begin{scope}[xshift=5cm]
\draw[gray] (0,0)--(0,4.8) (1,0)--(1,4.8) (3,0)--(3,4.8)
 (0,1.6)--(3,1.6) (0,1.8)--(3,1.8) (0,2.0)--(3,2.0)
(0,2.2)--(3,2.2) (0,2.4)--(3,2.4) (0,2.6)--(3,2.6)
(0,2.8)--(3,2.8) (0,3.0)--(3,3.0)
 (0,3.2)--(3,3.2) (0.2,1.6)--(0.2,3.2) (0.4,1.6)--(0.4,3.2) (0.6,1.6)--(0.6,3.2)
(0.8,1.6)--(0.8,3.2);
\draw[black, <-] (2,3.2)--(2,5.4) node[above] {medium cells};
\end{scope}
\end{scope}
\end{tikzpicture} \end{center} \caption{$S_\alpha$ and the cells the matching is used to fill (symbol $j$ goes in the $i$-th medium cell).} \label{salpha}
\end{figure}
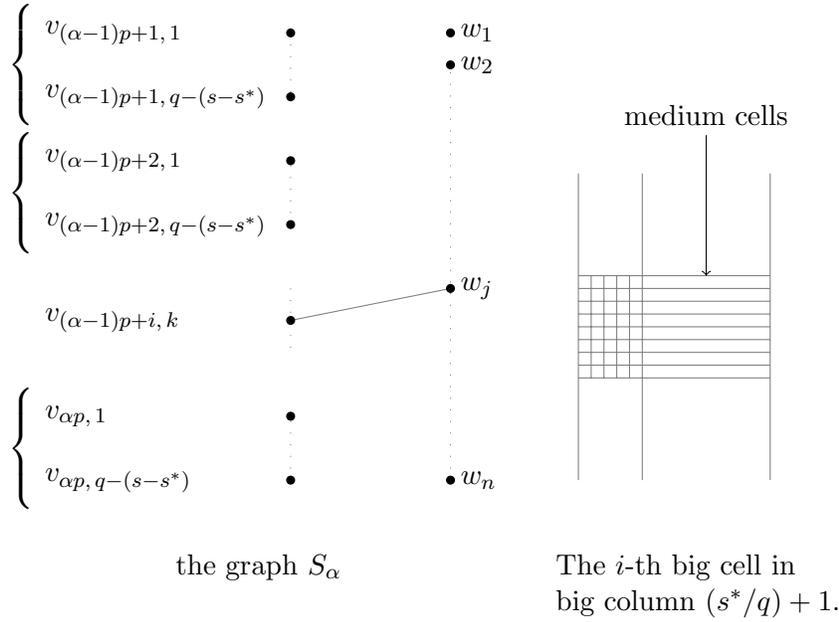

For $1 \le \alpha \le r^*/p$, if $S_\alpha$ has a matching $M_\alpha$, and if vertex $v_{(\alpha - 1)p + x, y}$ is joined to vertex $w_j$ by an edge of $M_\alpha$, then the symbol $j$ is placed in the $x$-th horizontal medium cell in the $\alpha$-th big row and the $((s^*/q) + 1)$-th big column. Exactly $q - (s - s^*)$ symbols are placed in this medium cell, and the symbols placed do not occur elsewhere in row $(\alpha - 1)p + x$. Moreover big cell $(\alpha, (s^*/q) + 1)$ contains each symbol exactly once. If $p \nmid r$ then the same technique is used to fill the $x$-th horizontal medium cell in row $r^* + x$ for $1 \le x \le r - r^*$ and big column $(s^*/q) + 1$, using the matching $M_{(r^*/p) + 1}$ in the graph $S_{(r^*/p) + 1}$. Now we apply the argument of Theorem~\ref{simpleryser} to fill all the remaining medium horizontal cells in big rows $1, \dots, r^*/p$, and the remaining medium horizontal cells in the first $r - r^*$ small rows in big row $(r^*/p) + 1$.

Next if $p \nmid r$ we apply the same argument using the matchings in the graphs $B_\beta$ ($1 \le \beta \le s^*/q)$ to fill in the vertical medium cells in big row $(r^*/p) + 1$ and big columns $1, 2, \dots, s^*/q$, and then to fill in the first $s - s^*$ vertical medium cells in big cell $((r^*/p) + 1, (s^*/q) + 1)$. (See Figure \ref{differentsizes}.) Finally, using the argument of Theorem~\ref{simpleryser} again, we fill all the remaining vertical cells in big columns $1, \dots, s^*/q$ and the first $s - s^*$ vertical medium cells in big column $(s^*/q) + 1$.

In the big cell $((r^*/p) + 1, (s^*/q) + 1)$ we have $(r - r^*) \times (s - s^*)$ single small cells, all filled, $s - s^*$ filled vertical medium cells of size $(p - (r - r^*)) \times 1$, $r - r*$ filled horizontal medium cells of size $1 \times (q - (s - s^*))$ and one empty cell of size $(p - (r - r^*)) \times (q - (s - s^*))$. We fill all the empty big cells with the symbols $1, 2, \dots, n$, and all the incompletely filled big cells with the symbols from $1, 2, \dots, n$ which do not already occur there.

One may easily check that we now have an $(S, T, U)$-outline latin square, where
\[ S = (1, 1, \dots, 1, p - (r - r^*), p, p, \dots, p) \]
\[ T = (1, 1, \dots, 1, q - (s - s^*), q, q, \dots, q) \]
\[ U = (1, 1, \dots, 1) \]
where $S$ has $r$ 1's and $q - (r^* + 1)$ $p$'s and $T$ has $s$ 1's and $p - (s^* + 1)$ $q$'s.

By Theorem \ref{alst}, there is a latin square $L$ of order $n$ of which our outline latin square is the $(S,T,U)$-amalgamation. Moreover, because the big cells all contain each of the symbols $1, 2, \dots, n$ exactly once, $L$ is a Sudoku square which extends $R_S$. This proves the necessity.

The sufficiency follows as the whole process can be reversed in a trivial way when $R_S$ is extended to a Sudoku square $L$. \end{proof}

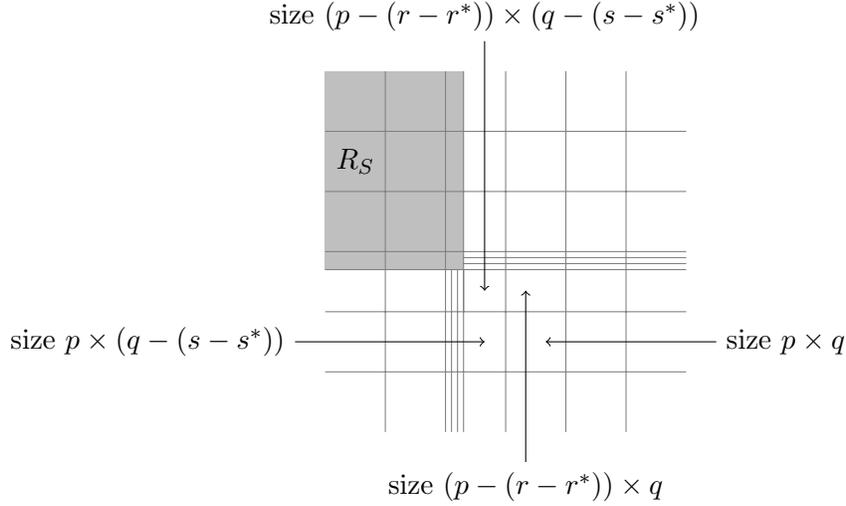
\begin{figure}
\begin{center} \begin{tikzpicture}
\begin{scope}[scale = 0.8]
\fill[gray!50] (0,1.7) rectangle (2.3, 5);
\draw[gray] (1,-1)--(1,5) (2,-1)--(2,5) (3,-1)--(3,5) (4,-1)--(4,5) (5,-1)--(5,5)
(2.3,1)--(2.3,5)
(0,0)--(6,0) (0,1)--(6,1) (0,2)--(6,2) (0,3)--(6,3) (0,4)--(6,4)
(2.1,-1)--(2.1,1.7) (2.2,-1)--(2.2,1.7) (2.3,-1)--(2.3,1.7)
(0,1.7)--(6,1.7) (2.3,1.8)--(6,1.8) (2.3,1.9)--(6,1.9);
\draw[black] (0.5,3.5) node {$R_S$};
\draw[black, <-] (3.333,1.35)--(3.333,-1.5) node[below] {size $(p - (r - r^*)) \times q$};
\draw[black, <-] (3.666,0.5)--(6.5,0.5) node[right] {size $p \times q$};

\draw[black, <-] (2.65,1.35)--(2.65,5.5) node[above] {size $(p - (r - r^*)) \times (q - (s - s^*))$};
\draw[black, <-] (2.65,0.5)--(-0.5,0.5) node[left] {size $p \times (q - (s - s^*))$};

\end{scope}
\end{tikzpicture} \end{center} \caption{Cells of different sizes.} \label{differentsizes}
\end{figure}

\section{Partial Sudoku Squares and \hc}

In this section we prove \threedash\ in which Theorem \ref{complicatedryser} is reformulated in terms of \hc\ for partial Sudoku squares. Let us remark that if $H$ is a subset of the set of cells of a partial Sudoku square, some of which are preassigned, then the \hi\ for $H$ is satisfied if and only if the \hi\ for the subset of $H$ consisting of its empty cells is satisfied.
This is easy to see, and is proved in slightly different circumstances in, for example, \cite{bghj}, \cite{hiltonvaughan} or \cite{hallstrength}. We shall just consider the most complicated case in \threedash\ where $p \nmid r$ and $q \nmid s$. The case when $p \mid r$ or $q \mid s$ is similar but easier; if $p \mid r$ and $q \mid s$ there is nothing to prove.

\begin{figure}
\begin{center} \begin{tikzpicture}
\begin{scope}[scale = 0.5]
\draw[gray] (0,0) rectangle (10,10)
(0,9)--(8,9) (6,10)--(6,3) (8,10)--(8,3)
(1,10)--(1,9) (2,10)--(2,9) (5,10)--(5,9)
(6,9)--(8,9) (6,8)--(8,8) (6,7)--(8,7) (6,3)--(8,3) (6,4)--(8,4);
\draw[black] (0.5,9.5) node {$1$} (1.5,9.5) node {$2$} (3.5,9.5) node {$\cdots$} (5.5,9.5) node {$x$}
(7,8.5) node {$x+1$} (7,7.5) node {$x+2$} (7,3.5) node {$n$};

\begin{scope}[xshift = 12cm]
\draw[gray] (0,0) rectangle (10,10)
(0,9)--(6,9)
(1,10)--(1,9) (2,10)--(2,9) (4,10)--(4,9) (6,10)--(6,9)
(6,9) rectangle (7,8) (7,8) rectangle (8,7) (9,6) rectangle (10,5);
\draw[gray, loosely dotted] (8,7)--(9,6);
\draw[black] (0.5,9.5) node {$1$} (1.5,9.5) node {$2$} (3,9.5) node {$\cdots$} (5,9.5) node {$x-1$}
(6.5,8.5) node {$x$} (7.5,7.5) node {$x$} (9.5,5.5) node {$x$};
\end{scope}
\end{scope}
\end{tikzpicture} \end{center} \caption{Two incomplete partial latin squares with $n$ cells preassigned.} \label{twoincomplete}
\end{figure}
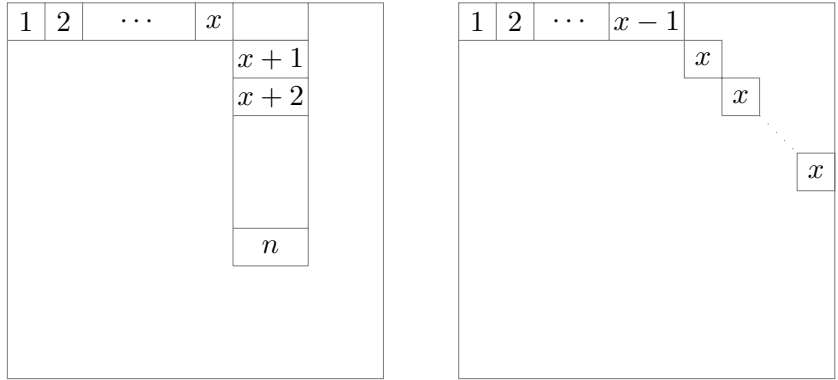

\begin{proof}[Proof of \threedashitalic.] By Theorem \ref{complicatedryser}, $R_S$ can be completed if and only if $R_S$ can be extended to an $(r^* + p) \times (s^* + q)$ \pqs\ rectangle, say $R_S^+$. So we need only focus on whether or not we can fill the remaining small cells in the $((s^*/q) + 1)$-th big column and $((r^*/p) + 1)$-th big row that already contain some preassigned cells. (We shall call the partially filled big cells the \textit{restricted big cells}.)

We know from Theorem \ref{complicatedryser} that these restricted big cells can be divided into appropriate horizontal and vertical medium cells if and only if the graphs $S_\alpha$ ($1 \le \alpha \le (r^*/p) + 1$) and $B_\beta$ ($1 \le \beta \le (s^*/q) + 1$) have matchings into the symbol vertices. (Note that the big cell $((r^*/p) + 1, (s^*/q) + 1)$ has only $r-r^*$ horizontal medium cells and $s-s^*$ vertical medium cells, and the corresponding matchings are smaller.)

By Hall's Theorem, a necessary and sufficient condition for the existence of such a matching $M_\alpha$ in $S_\alpha$ ($1 \le \alpha \le (r^*/p) + 1$) is that
\[ \bigcup_{v \in A} N(v) \ge \size{A} \]
for each subset $A$ of the set of (replicated) row vertices of $S_\alpha$ ($N(v)$ is the set of neighbours of a vertex $v$). But this is essentially the same as \hc\ for the $\alpha$-th restricted big cell in big column $(s^*/q) + 1$, namely
\[ \sum_{\sigma \in \{1, \dots, n\}} \alpha(L, \sigma, H) \ge \size{H} \]
for each subset $H$ of the restricted big cell. Here the $q - (s - s^*)$ replicates in $S_\alpha$ of a row vertex $v_{(\alpha - 1)p + j}$ in $S_\alpha^*$ correspond to the $q - (s - s^*)$ small cells in the $j$-th horizontal medium cell in the $\alpha$-th big cell in big column $(s^*/q) + 1$. A solution of just some of the row-vertex replicates corresponds to a subset $H$ containing just some of the small cells in the $j$-th medium cell. A symbol $\sigma$ is available for inclusion as an end-vertex of an edge of a matching, or in a cell of $H$, if and only if in at least one of the rows of $H$ there is no preassigned occurrence of the symbol. A symbol selected to be an end vertex of a matching may be selected to be placed in the corresponding row of $H$, and vice-versa. Thus in this case the two versions of \hc\ are essentially the same.

The argument for the matchings in the bottom graphs $B_\beta$ ($1 \le \beta \le (s^*/q) + 1$) is, of course, the same. \end{proof}

\section{Further comments}

\subsection{Analogues for Sudoku squares of results about partial latin squares}

\hc\ for graphs has been the object of a certain amount of study (see e.g. \cite{cropper00, cropper02, daneshgar, deh, eslahchijohnson, hiltonjohnson90b, hjw, hiltonvaughan, hallstrength}) and recently \hc\ for partial latin squares has been studied (see \cite{bghj, hiltonjohnson90a, hiltonvaughan, hallstrength}). In every case that has been looked at so far, whenever there is a theorem saying that a certain kind of partial latin square can be completed if and only if such and such a condition is satisfied, then the condition (whatever it might be) can be replaced by \hc. However, there are examples of classes of partial latin square for which \hc\ is not a sufficient condition for completion (it is always necessary).

It seems quite likely that there are analogues of these results for partial Sudoku squares. At present, we just have the example of Ryser's Theorem, given in this paper. It would also be interesting to construct classes of partial Sudoku squares for which \hc\ is not a sufficient condition for there to exist a completion.

It is worth remarking that there are really two natural kinds of partial Sudoku squares. In one kind the set of preassigned cells includes either none or all of the constituent small cells in any big cell (in the terminology of Section 3), and in the other this restriction is not made. Theorem \ref{simpleryser} considers partial Sudoku squares of the former type and Theorem \ref{complicatedryser} considers partial Sudoku squares of the latter type. In the next subsection we will look at two Sudoku variants of the Evans Conjecture.

\begin{figure}
\begin{center} \begin{tikzpicture}
\begin{scope}[scale=0.8]
\draw[gray,xstep=16mm, ystep=16mm] (0,-0.4) grid (6.4, 1.6);
\draw[black] (0.2,1.4) node {$1$} (0.6,1.4) node {$2$} (1.0,1.4) node {$3$} (1.4,1.4) node {$4$}
(1.8,1.0) node {$5$} (3.4,0.6) node {$5$} (5.0,0.2) node {$5$};
\end{scope}
\end{tikzpicture} \end{center} \caption{An incompletable partial Sudoku square.} \label{incompletepss1}
\end{figure}

\subsection{Two Sudoku analogues of the Evans Conjecture}

In 1960 Trevor Evans \cite{evans} posed the now famous conjecture that bears his name: What is the least number of preassigned elements an \ntn\ matrix can have and not be completable? The conjecture was $n$, and $n$ would be the best possible because of the configurations in Figure \ref{twoincomplete}. The Evans Conjecture was eventually proved independently by Smetaniuk \cite{smetaniuk} and Andersen and Hilton \cite{thankevans}. We can consider two variants of the Evans Conjecture for Sudoku squares. In the first, we ask:

\begin{problem} What is the least number of small cells that can be preassigned in a partial \ptq\ Sudoku square, so that it is not completable? \end{problem}

It seems reasonable to suppose that the answer is $p+q-1$, and this bound can be attained (see Figure \ref{incompletepss1}). (This question has also been asked by Antal Iv\'anyi \cite{ivanyi}.) Note that this includes the normal Evans Conjecture as a special case, as any latin square can be regarded as a $(1,n)$-Sudoku square.

The second variant of the Evans Conjecture for Sudoku squares, is one where the big cells are either empty or completely filled.



\begin{problem} What is the least number of big cells that can be preassigned in a partial \ptq\ Sudoku square, so that it is not completable? \end{problem}

We conjecture that the least number is $n=pq$. This bound can be attained by the following construction. Let $R_1, R_2, \dots, R_k$ be the rows of a $k \times k$ matrix $M$ filled with the symbols $1, 2, \dots, k^2$, each symbol occurring once. Let $M_1 = M$ and, for $2 \le i \le k$, let $M_i$ be the matrix
\[ M_i = \begin{pmatrix} R_i \\ R_{i+1} \\ \vdots \\ R_n \\ R_1 \\ \vdots \\ R_{i-1} \end{pmatrix} \]
Consider the partial $k^2 \times k^2$ Sudoku square in Figure \ref{incompletepss2}. Then, looking at the first row and $ik$-th column, we see that there is no symbol that can be placed in cell $(1, ik)$.

\begin{figure}
\begin{center} \begin{tikzpicture}
\begin{scope}[scale = 0.6]
\draw[gray] (0,0) rectangle (10,10)
(0,9)--(8,9) (6.5,10)--(6.5,3) (8,10)--(8,3)
(1,10)--(1,9) (2,10)--(2,9) (5,10)--(5,9)
(6.5,9)--(8,9) (6.5,8)--(8,8) (6.5,7)--(8,7) (6.5,3)--(8,3) (6.5,4)--(8,4);
\draw[black] (0.5,9.5) node {$M_1$} (1.5,9.5) node {$M_2$} (3.5,9.5) node {$\cdots$} (5.75,9.5) node {$M_{i-1}$}
(7.25,8.5) node {$M_i^\mathrm{T}$} (7.25,7.5) node {$M_{i+1}^\mathrm{T}$} (7.25,3.5) node {$M_n^\mathrm{T}$};

\end{scope}
\end{tikzpicture} \end{center} \caption{An incompletable partial Sudoku square.} \label{incompletepss2}
\end{figure}

\subsection{Orthogonal latin squares and \hc}

Suppose that $(A,B)$ is a pair of \ntn\ orthogonal latin squares on the symbols $1,2,\dots,n$. For $1 \le i \le n$, let $P_i$ be the set of cells of $B$ that contain the symbol $i$. Then, obviously, $(P_1, \dots, P_n)$ is a partition of the set of cells of an \ntn\ matrix such that $\size{P_i} = n$. Therefore $A$ is a $(P_1, \dots, P_n)$-Gerechte Design. In general, any latin square with an orthogonal mate can be thought  as a kind of Gerechte Design. Given a latin square $A$ with an associated partition $(P_1, \dots, P_n)$ corresponding to an orthogonal mate, we call $A$ a \textit{$(P_1, \dots, P_n)$-orthogonal latin square design}.

The partition $(P_1, \dots, P_n)$ itself has the properties:

\begin{enumerate}[(i)]
\item We can form a latin square $B$ by placing symbol $i$ in the cells of $P_i$ ($1 \le i \le n$).
\item There is a way of placing the symbols $1, 2, \dots, n$ each exactly once in the cells of $P_i$ ($1 \le i \le n$) so that we form another latin square $A$.
\end{enumerate}

Then, of course, $(A,B)$ is a pair of orthogonal latin squares. Any partition $(P_1, \dots, P_n)$ with properties (i) and (ii) is an \textit{OLS-partition of the \ntn\ matrix}. (OLS stands for ``orthogonal latin square'').

A partial $(P_1, \dots, P_n)$-OLS design is a partial latin square which is a partial $(P_1, \dots, P_n)$-Gerechte Design. Thus in a partial $(P_1, \dots, P_n)$-OLS design, we are given the partition $(P_1, \dots, P_n)$, and we have a partial latin square in which each $P_i$ contains each symbol at most once.

The notion of \textit{\hc\ for a partial $(P_1, \dots, P_n)$-OLS design} is just a particular special case of \hc\ for partial $(P_1, \dots, P_n)$-Gerechte Design defined in Section 1. We shall refer to it briefly as \textit{\hc\ for OLS-designs}.

It would be fascinating if there were results analogous to Ryser's Theorem for OLS-designs. Our theorem about Sudoku squares suggests the following possibility:

\begin{problem} Let $(P_1, \dots, P_n)$ be an OLS-partition of an \ntn\ matrix of cells. Suppose that we have an \ntn\ partial latin square $P_\vartheta$ in which $P_1, \dots, P_r$ are filled with each $P_i$ ($1 \le i \le r$) containing each symbol exactly once. Is it true that we can complete the partial latin square in such a way that each of $P_{r+1}, \dots, P_n$ also contain each symbol exactly once if and only if $P_\vartheta$ satisfies \hc\ for partial $(P_1, \dots, P_n)$-OLS designs? \end{problem}

\subsection{Complexity}

We do not know the complexity of the decision problem: given a partial latin square, decide if it satisfies \hc\ for latin squares. This problem is known to be in co-NP \cite{hiltonvaughan}, but it is not known to be in NP. We can ask the same question about the complexity of the corresponding problem for Sudoku squares: given a partial Sudoku square, decide if it satisfies \hc\ for Sudoku squares.

Colbourn \cite{colbourn} showed that the following decision problem is NP-complete: given a partial latin square, decide if it can be completed. Yato and Seta \cite{yatoseta} showed that the corresponding problem for Sudoku squares is also NP-complete: given a partial Sudoku square, decide if it can be completed. We do not know if the problem remains NP-complete if the big cells are required to be either filled or empty, although it seems reasonable to suppose that this is indeed the case.

In \cite{vaughan} it is shown that deciding if a partial Gerechte Design is completable is NP-complete, even if the partial Gerechte Design contains no filled cells at all.

Hilton and Vaughan \cite{hiltonvaughan} have looked at the decision problem: given a partial latin square satisfying \hc\ for latin squares, decide if it can be completed. They showed that this problem is NP-hard. We do not know if this is true of the corresponding problem for Sudoku squares: given a partial Sudoku square satisfying \hc\ for Sudoku squares, decide if it can be completed.

A maximum matching in a bipartite graph on $n$ vertices can be found in $O(n^3)$ time (see e.g. \cite[Chapter 20]{schrijver}), so there 
is a polynomial algorithm to decide if a bipartite graph has a perfect matching. To decide if an \rts\ partial Sudoku square, as considered in Theorem \ref{complicatedryser}, can be completed, we need to decide if each of $n$ bipartite graphs has a perfect matching. Consequently, the problem of deciding if such a partial Sudoku squares can be completed can be performed in polynomial time.


\begin{thebibliography}{9}

\bibitem{andersenhilton1} L. D. Andersen, A. J. W. Hilton, Generalized Latin rectangles. I. Construction and decomposition, \textit{Discrete Math.}, \textbf{31} (1980), 125--152.

\bibitem{andersenhilton2} L. D. Andersen, A. J. W. Hilton, Generalized Latin rectangles. II. Embedding, \textit{Discrete Math.}, \textbf{31} (1980), 235--260.

\bibitem{thankevans} L. D. Andersen, A. J. W. Hilton, Thank Evans!, \textit{Proc. London Math. Soc. (3)}, \textbf{47} (1983), 507--522.

\bibitem{bcc} R. A. Bailey, P. J. Cameron and R. Connelly. Sudoku, gerechte designs, resolutions, affine space, spreads, reguli, and Hamming codes, \textit{Amer. Math. Monthly} \textbf{115} (2008), 383--404.

\bibitem{bghj} B. B. Bobga, J. Goldwasser, A. J. W. Hilton and P. D. Johnson, Jr., Completing partial latin squares: Cropper's Question, \textit{Australas. J. Combin.}, \textbf{49} (2011), 127--152.


\bibitem{colbourn} C. J. Colbourn, Complexity of completing partial Latin squares, \textit{Disc. Appl. Math.}, \textbf{8} (1984), 25--30.


\bibitem{cropper00} M. M. Cropper, J. L. Goldwasser, A. J. W. Hilton, D. G.Hoffman and P. D.Johnson, Jr., Extending the disjoint-representatives theorems of Hall, Halmos, and Vaughan to list-multicolorings of graphs, \textit{J. Graph Theory}, \textbf{33} (2000), 199--219.

\bibitem{cropper02} M. M. Cropper and A. J. W. Hilton, Hall parameters of complete and complete bipartite graphs, \textit{J. Graph Theory}, \textbf{41} (2002), 208--237.

\bibitem{daneshgar} A. Daneshgar, A. J. W. Hilton and P. D. Johnson, Jr., Relations among the fractional chromatic, choice, Hall, and Hall-condition numbers of simple graphs, \textit{Discrete Math.}, \textbf{241} (2001), 189--199.

\bibitem{deh} J. K. Dugdale, Ch. Eslahchi and A. J. W. Hilton, The Hall-condition index of a graph and the overfull conjecture, \textit{J. Combin. Math. Combin. Comput.}, \textbf{35} (2000), 197--216.

\bibitem{edmonds} J. Edmonds, Paths, trees, and flowers. \textit{Canad. J. Math.}, \textbf{17} (1965), 449--467.

\bibitem{egervary} E. Egerv\'ary, On combinatorial properties of matrices, \textit{Mat. Lapok.}, \textbf{38} (1931), 16--28.

\bibitem{eslahchijohnson} Ch. Eslahchi and M. Johnson, Characterization of graphs with Hall number 2, \textit{J. Graph Theory}, \textbf{45} (2004), 81--100.

\bibitem{evans} T. Evans, Embedding incomplete latin squares. \textit{Amer. Math. Monthly}, \textbf{67} (1960), 958--961.

\bibitem{ghp} J. Goldwasser, A. J. W. Hilton and D. Patterson, Cropper's question and Cruse's theorem about partial Latin squares. \textit{J. Combin. Des.}, \textbf{19} (2011), 268--279.

\bibitem{hall} P. Hall, On representatives of subsets, \textit{J. London Math. Soc.}, \textbf{10} (1935), 26--30.

\bibitem{hilton80} A. J. W. Hilton, The reconstruction of Latin squares with applications to school timetabling and to experimental design, \textit{Math. Programming Stud.}, \textbf{13} (1980), 68--77.

\bibitem{hilton87} A. J. W. Hilton, Outlines of Latin squares, \textit{Ann. Discrete Math.}, \textbf{34} (1987), 225--242.

\bibitem{hiltonjohnson90a} A. J. W. Hilton and P. D. Johnson, Jr., A variation of Ryser's theorem and a necessary condition for the list-colouring problem, in R. Nelson and R. Wilson (eds.), \textit{Graph colourings}, Longman, 1990.

\bibitem{hiltonjohnson90b} A. J. W. Hilton and P. D. Johnson, Jr., Extending Hall's theorem, in \textit{Topics in combinatorics and graph theory}, Physica, Heidelberg, 1990, 359–371.

\bibitem{hiltonjohnson99} A. J. W. Hilton and P. D. Johnson, Jr., The Hall number, the Hall index, and the total Hall number of a graph, \textit{Discrete Appl. Math.}, \textbf{94} (1999), 227–245.

\bibitem{hjw} A. J. W. Hilton, P. D. Johnson, Jr., and E. B. Wantland, The Hall number of a simple graph, \textit{Congr. Numer.}, \textbf{121} (1996), 161--182.

\bibitem{hallstrength} A. J. W. Hilton and E. R. Vaughan, The Hall Strength and the Conjugate Hall Strength of a Partial Latin Square, Proceedings of the Forty-First Southeastern International Conference on Combinatorics, Graph Theory and Computing, \textit{Congr. Numer.}, \textbf{205} (2010), 75--96.

\bibitem{hiltonvaughan} A. J. W. Hilton and E. R. Vaughan, Hall's Condition for Partial Latin Squares, submitted.

\bibitem{hoffmanjohnson} D. G. Hoffman and P. D. Johnson, Jr., Extending Hall's theorem into list colorings: a partial history, \textit{Int. J. Math. Math. Sci.}, 2007.

\bibitem{ivanyi} A. Iv\'anyi, problem session at the Building Bridges conference, Budapest, August 2008.

\bibitem{mcdiarmid} C. J. H. McDiarmid, The solution of a timetabling problem, \textit{J. Inst. Math. Appl.}, \textbf{9} (1972), 23--34.

\bibitem{ryser} H. J. Ryser, A combinatorial theorem with an application to latin rectangles, \textit{Proc. Amer. Math. Soc.}, \textbf{2} (1951), 550--552.

\bibitem{schrijver} A. Schrijver, \textit{Combinatorial Optimization}, Springer, 2003.

\bibitem{smetaniuk} B. Smetaniuk, A new construction on Latin squares. I. A proof of the Evans conjecture, \textit{Ars Combin.}, \textbf{11} (1981), 155--172.

\bibitem{vaughan} E. R. Vaughan, The complexity of constructing gerechte designs, \textit{Electron. J. Combin.}, \textbf{16} (2009), R15.

\bibitem{dewerra} D. de Werra, A few remarks on chromatic scheduling, in B. Roy (ed.), \textit{Combinatorial Programming: Methods and Applications}, Reidel, Dordrecht, 1975, 337--342.

\bibitem{yatoseta} T. Yato and T. Seta, Complexity and Completeness of Fnding Another Solution and Its Application to Puzzles, \textit{IEICE Transactions on Fundamentals of Electronics, Communications and Computer Sciences}, \textbf{E86-A} (2003), 1052--1060.

\end{thebibliography}
\end{document}